\documentclass[12pt,reqno]{amsart}

\usepackage{epsfig}
\usepackage{amsmath, amsthm, amsfonts, amssymb, mathrsfs}
\usepackage{graphicx}

\numberwithin{equation}{section}

\renewcommand{\Re}{{\operatorname{Re\,}}}
\renewcommand{\Im}{{\operatorname{Im\,}}}

\newcommand{\wh}{\widehat}

\newcommand{\Tr}{{{\operatorname{Tr}}}}

\newtheorem{theo}{{\sc \bf Theorem}}[section]
\newtheorem{cor}[theo]{{\sc \bf Corollary}}

\newtheorem{prop}[theo]{{\sc \bf Proposition}}

\newenvironment{rem}{\medskip\noindent{\it Remark:\/} }{\medskip}

\begin{document}

\title{A note on spectral properties of the $p$-adic tree}

\author{Slawomir Klimek}
\address{Department of Mathematical Sciences,
Indiana University-Purdue University Indianapolis,
402 N. Blackford St., Indianapolis, IN 46202, U.S.A.}
\email{sklimek@math.iupui.edu}

\author{Sumedha Rathnayake}
\address{Department of Mathematical Sciences,
Indiana University-Purdue University Indianapolis,
402 N. Blackford St., Indianapolis, IN 46202, U.S.A.}
\email{srathnay@iupui.edu}

\author{Kaoru Sakai}
\address{Department of Mathematical Sciences,
Indiana University-Purdue University Indianapolis,
402 N. Blackford St., Indianapolis, IN 46202, U.S.A.}

\date{\today}\email{ksakai@iupui.edu} 

\begin{abstract}
We study the spectrum of the operator $D^*D$, where the operator $D$, introduced in \cite{KMR}, is a forward derivative on the $p$-adic tree, a weighted rooted tree associated to $\mathbb Z_p$ via Michon's correspondence. We show that the spectrum is closely related to the roots of a certain $q-$hypergeometric function and discuss the analytic continuation of the zeta function associated with $D^*D$.
\end{abstract}

\maketitle
\section{Introduction}

This note builds on our previous paper \cite{KMR} which described a new spectral triple $(\mathcal A, \mathcal H,\mathcal D)$ for the $C^*$-algebra of continuous functions on the space of $p$-adic integers $\mathbb Z_p$. The construction of this spectral triple utilized a coarse-grained approximation of the space $\mathbb Z_p$ and was partially motivated by recent work \cite{BP} on  spectral triples for more general Cantor sets.  Our considerations  closely resembled standard examples of geometric spectral triples that use the usual differentiation for the definition of the operator $\mathcal D$. 

The geometric coarse-grained approximation of $\mathbb Z_p$, which we called the $p$-adic tree, is a weighted rooted tree  $\{V,E\}$ associated to $\mathbb Z_p$ via Michon's correspondence \cite{BP}. The set of vertices $V$ of the $p$-adic tree consist of balls in $\mathbb Z_p$, with $\mathbb Z_p$ itself being the root of the tree. There is an edge between two vertices $v$ and $v'$ if $v' \subset v$ and $v'$ has the biggest diameter smaller than the diameter of $v$. 

Now consider the Hilbert space $H$ consisting of weighted $\ell^2$ functions living on the vertices of the $p$-adic tree:
\begin{equation}\label{Hilbdef}
H= \{f: V \rightarrow \mathbb C : \sum_{v \in V}|f(v)|^2 w(v)< \infty\}.
\end{equation}
Here the discrete-valued weight function $w: V \rightarrow \mathbb R_{\geq 0}$ is defined by: $w(v) = $ the volume of the ball $v$ with respect to the additive Haar measure $d_px$. For the space of $p$-adic integers $\mathbb Z_p$, the volume of a ball is equal to its diameter. It is useful to view $H$ as the subspace of $L^2(\mathbb Z_p,d_px)$ consisting of locally constant functions on $\mathbb Z_p$.

Then we introduced an unbounded operator $D$ on $H$ defined on its maximal domain $\{f\in H : Df\in H\}$ by 
\begin{equation}\label{Ddef}
Df(v)=\frac 1{\omega(v)}\left(f(v)-\frac 1{(\textrm{deg } v-1)} \sum_{\substack {v'\in V \\ v'\sqsubset\ v}}f(v')\right),
\end{equation}
where deg $v$ is the degree of the vertex $v$ and $v' \sqsubset v$  means that there is an edge between $v$ and $v'$. One can think of $D$ as a natural discrete derivative for (complex valued) functions on $\mathbb Z_p$.
This forward tree derivative  was then used to construct the Dirac type operator $\mathcal D$, necessary for the spectral triple. 

It was verified in \cite{KMR} that the operator $D$ is invertible with compact inverse, implying that $D^*D$ has compact resolvent. Consequently, the spectrum of $D^*D$ is discrete with only possible accumulation point at infinity. 

In this paper our main interest is to find the spectrum of the operator $D^*D$. By re-parametrizing the vertices of the $p$-adic tree using the set of parameters $V \cong \mathcal G_p \times \mathbb Z_{\geq 0}$, where $\mathcal G_p=\mathbb Q_p / \mathbb Z_p$ is the Pr\"ufer $p$-group, we can decompose $H$ into invariant subspaces: $H=\bigoplus_{g\in \mathcal G_p} H_g$, where $H_g\cong \ell^2(\mathbb Z_{\geq 0})$. This allows the decomposition of the operators $D$ and $D^*$ into a direct sum of much simpler operators
$D=\bigoplus_{g\in \mathcal G_p} D_g$, and $D^*=\bigoplus_{g\in \mathcal G_p} D_g^*$,
where $D_g$ is the restriction of $D$ to $H_g$. Identifying the Pr\"ufer group with the set of numbers $\{ g=\frac r{p^m} :\; 0 \leq r < p^m, p\nmid r \}$, the operators $D_g, D_g^*$ for $g=\frac r{p^m}$ can be written as $D_g=p^m D_0$ and $D_g^*=p^m D_0^*$ where $D_0$ is the operator on $\ell^2(\mathbb Z_{\geq 0})$ given by $
D_0 f_n= p^n (f_n-f_{n+1})$.
Consequently, $D^*D=\bigoplus_{g \in \mathcal G_p} p^{2m} D_0^*D_0$ and the problem of finding the spectrum of $D^*D$ is reduced to finding the spectrum of $D_0^*D_0$. 

It will be verified in this paper that the eigenvalues of $D_0^* D_0 $ are the roots of the $q$-Bessel function $_1\phi_1\left(\substack {0\\q};q,\lambda \right)$ with $q=p^{-2}$. In \cite{ABC} the authors give analytic bounds for these roots and discuss their asymptotic behavior. Therefore, we have a good understanding of the spectrum of $D^*D$. In particular, using the results of \cite{ABC}, we were able to obtain several results on analytic structure and analytic continuation of the zeta function of  $D^*D$.

Part of the motivation for studying the spectrum of $D^*D$ is that it may have some relevance for developing the structure of $p$-adic quantum mechanics. The operator $D^*D$, a natural analog of the laplacian, can be taken as an alternative starting point for the theory of $p$-adic Schr$\ddot{ \textrm{o}}$dinger operators, see \cite{VVZ}. 

The content of this paper is organized as follows. In section 2 we give a brief introduction to $p$-adic harmonic analysis and then describe the $p$-adic tree associated to the ring $\mathbb Z_p$. In the next section we analyze the forward derivative $D$ on the $p$-adic tree and its adjoint $D^*$ in $H$. We also describe the re-parametrization of the $p$-adic tree that leads to a decomposition of the operator $D^* D$ into a direct sum of simpler operators $D_g^* D_g $. Section 4 discusses the calculation of the spectrum of $D_0^* D_0 $ along with an ``elliptic regularity" theorem that shows that the corresponding eigenfunctions are of a special form. In the last section we discuss some spectral properties of the operator $D^* D$ and the analytic continuation of the zeta function associated to it.  
\bigskip


\section{Definitions and Notation}

\subsection{Fourier Analysis in $\mathbb Z_p$} 
We start this section by briefly recalling some of the basic results and  notation we introduced in \cite{KMR} regarding harmonic analysis in the space $\mathbb Z_p$. For further reference and more details on this subject we refer to \cite{HR},\cite{KMR} and \cite{VVZ}.
\smallskip

The characters on $\mathbb Q_p$, the set of all $p$-adic numbers, are given by maps $\chi_a: \mathbb Q_p \rightarrow \mathbb C$ defined by $\chi_a(x)=e^{2\pi i \{ax\}}$ where $a  \in \mathbb Q_p$, and $\{ax\}$ in the exponent is the fractional part of the p-adic number $ax$. Two characters $\chi_a(x)$ and $\chi_b(x)$ are equal on ${\mathbb Z_p}$ if and only if $a-b\in \mathbb Z_p$. Consequently we see that the dual groups of $\mathbb Q_p$ and $\mathbb Z_p$ (denoted $\widehat {\mathbb Q_p}$, $\widehat {\mathbb Z_p}$ ) are $\widehat {\mathbb Q_p}=\mathbb Q_p$ , $\widehat {\mathbb Z_p}= \mathbb Q_p / \mathbb Z_p$. The dual group of $\mathbb Z_p$, called the Pr\"ufer group $\mathcal G_p$, can also be identified with a group of roots of unity given by 
\begin{equation}\label{Prufdef}
\wh {\mathbb Z_p} \cong \{e^{2 \pi i \frac {k}{p^n}} :  n \in \mathbb Z_{\geq 0}, p \nmid k \in \mathbb Z\}.
\end{equation}

We let $\mathcal E(\mathbb Z_p)$ be the space of locally constant functions (test functions), i.e. the set of functions $\phi: \mathbb Z_p \rightarrow \mathbb C$ such that for every $x \in \mathbb Z_p$ there is a neighborhood $U_x$ of $x$ on which $\phi$ is constant. The space of linear functionals on $\mathcal E(\mathbb Z_p)$ (distributions on $\mathbb Z_p$) equipped with the weak$^*$-topology is denoted by $\mathcal E^*(\mathbb Z_p)$.

 If $d_px$ denotes the Haar measure on $(\mathbb Z_p, +)$ normalized so that $\int_{\mathbb Z_p}d_px=1$, then we define the Fourier transform of a test function $\phi \in \mathcal E(\mathbb Z_p)$ as the function $\wh \phi$ on $\wh {\mathbb Z_p}$ given by 
 $$ \wh \phi([a])=\int_{\mathbb Z_p} \phi(x) \overline{\chi_a(x)}d_px.$$ 
 For a locally constant function, only a finite number of Fourier coefficients will be nonzero. Thus, the Fourier transform gives an isomorphism between $\mathcal E(\mathbb Z_p)$ and $\mathcal E(\wh{ \mathbb Z_p})$, where the latter in our case is the space of all those functions on $\wh{ \mathbb Z_p}$ that are zero almost everywhere. The inverse Fourier transform is given by 
 $$\phi(x)= \sum_{[a]\in \wh {\mathbb Z_p}}  \wh \phi([a]) \chi_a(x).$$
 
 For a distribution $T \in \mathcal E^*(\mathbb Z_p)$ the Fourier transform is the function $\wh T$ on $\wh {\mathbb Z_p}$ defined by $\wh T([a])= T(\overline{\chi_a(x)})$. Once again, the distributional Fourier transform gives an isomorphism between $\mathcal E^*(\mathbb Z_p)$ and $\mathcal E^*(\wh {\mathbb Z_p})$. The inverse Fourier transform of a distribution is given by $$T(\chi_a(x))= \sum_{[a]\in \wh {\mathbb Z_p}} \wh T([a]) \chi_a(x).$$ 
 \bigskip
 
\subsection{The $p$-adic tree}
We recall the construction of the weighted rooted tree $\{V,E\}$ associated to the Cantor metric space $(\mathbb Z_p, \rho_p)$, the space of $p$-adic integers equipped with the usual $p$-adic metric $\rho_p$, via Michon's correspondence \cite{BP}, \cite{KMR}. The symbols $V$ and $E$ above are used to denote the set of vertices and the edges of the  tree, and we call this tree the $p$-adic tree. 
\smallskip

The vertices of the $p$-adic tree are the balls in $\mathbb Z_p$. Since $\mathbb Z_p$ is a totally disconnected space, the range of $\rho_p$ is countable and consists of numbers of the form $p^{-n}$, $n \in \mathbb Z$ and zero.  Therefore, if we let $V_n$ be the set of balls of diameter $p^{-n}$ then the set of vertices $V$ has the natural decomposition $V= \bigcup_{n=0}^{\infty} V_n$. The set of edges $E$ has the decomposition $E= \bigcup_{n=0}^{\infty} E_n$ where an edge $e=(v,v')$ between two vertices $v, v'$ belongs to $E_n$ if  $v \in V_n$, $v' \in V_{n+1}$ and $v' \subset v$. 

The compactness of the space $\mathbb Z_p$ implies that the number of balls of diameter $p^{-n}$ (hence the number of vertices) for fixed $n$, and the degree of each vertex are finite. Now we observe the following fact:

\begin{prop}
Every ball of radius $p^{-n}$ contains a unique integer $k$ such that $0 \leq k < p^n$.
\end{prop} 
 A proof of this proposition can be found in \cite{KMR}. From this observation we see that there is a one-to-one correspondence between the set of integers $0 \leq k < p^n$ and the set $V_n$ of balls of diameter $p^{-n}$. Therefore, the set of vertices has the natural parametrization: 
 \begin{equation}\label{para}
 V\cong S:=\{(n,k) : n=0,1,2,\ldots, \;0 \leq k < p^n\}.
 \end{equation}
Two vertices $(n,k)$ and $(n+1,k')$ are connected by an edge if and only if $k'-k$ is divisible by $p^n$. Thus, a given vertex $(n,k)$ will be connected (via edges) to exactly $p$ vertices in $V_{n+1}$. Also, we introduce a weight function $w: V \rightarrow \mathbb R^+$ by $w(v)=$ volume $(v)$ with respect to the Haar measure $d_px$. If $v\in V_n$ then $w(v)=p^{-n}$.
\bigskip 
 
 
\section{The Operator $D$}

\subsection{A Forward Derivative on the $p$-adic tree }
Due to the decomposition $V= \bigcup_{n=0}^{\infty} V_n$ of the set of vertices, any complex valued function $f$ on $V$ can be written as a sequence  $\{f_n\}$ of complex valued functions on $V_n$. Let $\mathcal E^*(V)$ denote the space of all complex valued functions living on the vertices of the $p$-adic tree. By the discussion in the previous section we can identify each $V_n\cong \mathbb Z_p / p^n \mathbb Z_p \cong \mathbb Z / p^n \mathbb Z$ with the finite additive group $\mathbb Z / p^n \mathbb Z$. Consequently, we can introduce the Fourier transform of a function $f \in \mathcal E^*(V)$ to be the discrete Fourier transform on each $V_n$ given by 
\begin{equation*}
\wh f_n(l)= \frac 1{p^n} \sum_{k=0}^{p^n-1}f_n(k)e^{-2 \pi i \frac{kl}{p^n}}, \;\; 0 \leq l < p^n.
\end{equation*}

Because the characters of $\mathbb Z / p^n \mathbb Z$ satisfy the orthogonality condition: 
\begin{equation}\label{orthocon}
\sum_{0\leq s<p^j}e^\frac{-2\pi iks}{p^j}=\begin{cases}
0 & \textrm{ if }\ p^j\nmid k \\ 
p^j & \textrm{ if }\ p^j\mid k,
\end{cases}
\end{equation}
we obtain the following Fourier inversion formula:
\begin{equation}\label{FIof_f}
f_n(k)=\sum_{0\leq l<p^n}\wh f_n(l)e^{\frac{2\pi ikl}{p^n}}.
\end{equation}

We also remark that the $p$-adic tree can be thought to be self dual, $\wh V\cong V$, due to the fact that each $\mathbb Z / p^n \mathbb Z$ (hence $V_n$) is self dual. Thus, the Fourier transform on the $p$-adic tree is an isomorphism between the space  $\mathcal E^*(V)$ of functions on the vertices of the $p$-adic tree and the space  $\mathcal E^*(\wh V)$ of functions on the vertices of the dual tree. Additionally, via the Parseval's identity, the Fourier transform gives an isomorphism between the Hilbert space $H= \ell^2(V,w)$, of \eqref{Hilbdef}, and 
$\wh H:= \ell^2(\wh V)$, where the latter Hilbert space has no weight in the inner product.

%
 
Notice that the decomposition $V= \bigcup_{n=0}^{\infty} V_n$ induces a Hilbert space decomposition $H= \bigcup_{n=0}^{\infty}\ell^2(V_n, p^{-n})$. Using the parametrization \eqref{para} introduced in the previous section, we can write the action of the operator $D$ of the formula \eqref{Ddef} on the components $f_n$ of $f$ as:

\begin{equation*}
 (Df)_n(k)=p^n \left(f_n(k)-\frac 1p \sum_{0 \leq j < p} f_{n+1}(k+jp^n)\right).
\end{equation*} 
We note here that the choice of the domain for $D$, as well as all other unbounded operators below, is the maximal domain, i.e $\{f\in H : Df\in H\}$.

Using the Fourier transform of $f_n$ and orthogonality of the characters \eqref{orthocon} we write the following equivalent formula for $D$:

\begin{equation*}
 Df_n(k)=p^n \sum_{0 \leq l < p} \left(\wh f_n(l)- \wh f_{n+1}(pl)\right) e^{\frac{2 \pi i kl}{p^n}}.
\end{equation*} 
Hence, in Fourier transform, the operator $D$ becomes $\wh D$ given by:
\begin{equation}\label{Dhat}
\wh D \wh f_n(l)= p^n (\wh f_n(l)- \wh f_{n+1}(pl)),
\end{equation}
which is an unbounded operator on $\wh H$.
Notice that $D$ and $\wh D$ are unitarily equivalent, but it is easier to work with the latter.

The adjoint $D^*$ of $D$ is given by
\begin{equation}\label{D*}
D^*g_n(k)=p^n\left[g_n(k)-\frac 1p g_{n-1}(k\mod p^{n-1})\right],
\end{equation}
assuming $g_{-1}(0)=0$. 

Later we will need the following formula for the adjoint $\wh {D^*}$ of $\wh D$:

\begin{equation}\label{Dhat*}
\wh {D^*} \wh g_n(l)=\begin{cases}
p^n \wh g_n(l)& \textrm{ if }\ p\nmid l \\ 
p^n\left(\wh g_n(l)- \frac 1p \wh g_{n-1}\left(\frac lp\right)\right) & \textrm{ otherwise.}
\end{cases}
\end{equation} 

It was verified in \cite{KMR} that $D$ and $D^*$ are invertible with compact inverses.
\smallskip

\subsection{Invariant Subspaces of $H$}

The key observation that allows us to find the spectrum of $D^*D$ is that we can decompose the Hilbert space $H$ into invariant subspaces by means of a different parametrization of the $p$-adic tree. 
\smallskip

The original parametrization \eqref{para} of the set of vertices of $p$-adic tree was done by using the set
 \begin{equation*}
 S:=\{(n,k) :  n=0,1,2,\ldots, 0 \leq k < p^n\}.
 \end{equation*}
Given a pair $(n,k)$ in $S$ notice that we can write $k=rp^l$ with $p\nmid r$ and $l\in \{0,1, \ldots n-1\}$, by factoring out the highest power of $p$ that divides $k$. Such a representation of $k$ will be uniquely determined by $r$ and $l$. If we associate $n$ with $\frac{k}{p^n}=\frac r{p^{n-l}}= \frac r{p^m}$ where $m=n-l$, which is a unique representation of $n$ in terms of $r$ and $m$, then we have the correspondence $(n,k) \mapsto \left(\frac r{p^m},l\right)$.

Conversely, given a pair $\left(\frac r{p^m},l\right)$ where $0 \leq r < p^m$, $p \nmid r$ and $l \in \{0,1,2, \ldots\}$ we can make the unique association $\left(\frac r{p^m},l\right) \mapsto (m+l, rp^l)$.
Thus, if
\begin{equation*}
S':= \left\{\left(\frac r{p^m},l\right) : 0 \leq r < p^m,\ p\nmid r,\ l=0,1,2,\ldots \right\}
\end{equation*}
then we have the one-to-one correspondence between the sets $S$ and $S'$ given by $(n,k) \leftrightarrow \left(\frac r{p^m}, l\right)$. 

In fact, the set of numbers $\left\{ g=\frac r{p^m} :\; 0 \leq r < p^m, p\nmid r \right\}$ is isomorphic to the Pr\"ufer group $\mathcal G_p$ defined in (\ref{Prufdef}). Therefore, $V\cong \wh V \cong \mathcal G_p \times \mathbb Z_{\geq 0}$. Consequently we obtain the following new decomposition of the Hilbert space $\wh H$:

\begin{equation*}
\wh H= \ell^2(S)\cong \ell^2(S')=\bigoplus_{\frac r{p^m}\in \mathcal G_p} \ell^2(\mathbb Z_{\geq 0})=: \bigoplus_{g\in \mathcal G_p} \wh H_g
\end{equation*}
where $\wh H_g=\ell^2(\mathbb Z_{\geq 0})$.

\smallskip

We will now look at the operators $\wh D$ and $\wh {D^*}$ in the new coordinates.  Using formula \eqref{Dhat} we compute:
 \begin{equation*}
  \wh D \wh f \left(\frac r{p^m},l\right)= p^{m+l} \left(\wh f\left(\frac r{p^m},l\right)-\wh f\left(\frac r{p^m},l+1\right)\right).
  \end{equation*}
 
 Equation \eqref{Dhat*} yields:
  \begin{equation*}
 \begin{aligned}
  \wh {D^*}\wh f\left(\frac r{p^m},l\right )&=\begin{cases}
  p^{m+l} \wh f\left(\frac r{p^m},0\right) & \textrm{ if } l=0 \\ 
  p^{m+l} \left(\wh f\left(\frac r{p^m},l\right)-\frac 1p \wh f \left(\frac r{p^m}, l-1\right)\right) & \textrm{ otherwise }.
 \end{cases}\\
 \end{aligned}
 \end{equation*}
 
If we assume that $\wh f \left(\frac r{p^m}, -1\right)=0$ for any $r,m$ then we can rewrite the formula for  $\wh {D^*}$ as:
\begin{equation}\label{NDhat*}
\wh {D^*}\wh f\left(\frac r{p^m},l\right )=p^{m+l} \left(\wh f\left(\frac r{p^m},l\right)-\frac 1p \wh f \left(\frac r{p^m}, l-1\right)\right).
\end{equation}
Notice that, in the new coordinates, the operators $\wh D$ and $\wh{D^*}$ affect only the second coordinate $l$ and consequently each $H_g$ is an invariant subspace. Thus, by letting $\widehat D_g:= \widehat D \vert_{H_g}$ and  $\widehat {D_g^*}:= \widehat {D^*} \vert_{H_g}$ of $\widehat {D_g}$, we have the following decompositions of the operators $\widehat D$ and $\widehat {D^*}$:
\begin{equation}\label{Dg}
\widehat D=\bigoplus_{g\in \mathcal G_p} \widehat {D_g} \textrm{\,\, and \,\,} \widehat {D^*}=\bigoplus_{g\in \mathcal G_p} \widehat {D_g^*}.
\end{equation}
\smallskip

Let $\widehat {D_0}$ be the operator on $\ell^2(\mathbb Z_{\geq 0})$ given by $\widehat {D_0}f(l)=p^l\left(f(l)-f(l+1)\right)$. It will be more convenient to switch to subscript notation and write:
\begin{equation}\label{D_0def}
(\widehat {D_0}f)_n=p^n \left(f_n-f_{n+1}\right).
\end{equation}
The adjoint of $\widehat {D_0}$ is given by
\begin{equation}\label{D_0*def}
(\widehat {D_0^*g})_n=p^n \left(g_n-\frac 1p g_{n-1}\right).
\end{equation}
From formula \eqref{NDhat*} we see that if $g=\frac r{p^m}$ then 
\begin{equation*}
\widehat {D_g}=p^m \widehat {D_0} \textrm{ \,\,and \,\,} \widehat {D_g^*}= p^m \widehat {D_0^*}.
\end{equation*}
Consequently, $\widehat {D^*}\widehat D$ has the decomposition;
\begin{equation}\label{decomD*D}
\widehat {D^*}\widehat D=\bigoplus_{g\in \mathcal G_p} \widehat {D_g^*}\widehat {D_g}=\bigoplus_{g\in \mathcal G_p} p^{2m} (\widehat {D_0^*}\widehat {D_0}).
\end{equation}

Thus, the key step in finding the spectrum of $D^*D$ is to compute the spectrum of the operator  $ D_0^* D_0$ on $\ell^2(\mathbb Z_{\geq 0})$. We devote the next section to a description of this spectrum. 

\bigskip

\maketitle
\section{Spectrum of $D_0^*D_0$}

The fact that $D^{-1}$ is compact implies that the operators $D^*D$ and $ D_0^* D_0$ have compact resolvent. Consequently, the spectrum of the unbounded operator $ D_0^* D_0$ consists of eigenvalues diverging to infinity.

Using formulas \eqref{D_0def} and \eqref{D_0*def} we obtain the following system of equations for $\widehat {D_0^*}\widehat {D_0}$.
\begin{equation*}
\begin{aligned}
(\widehat {D_0^*}\widehat {D_0}f)_0&=f_0-f_1\\
(\widehat {D_0^*} \widehat {D_0}f)_n&=p^n \left((\widehat {D_0}f)_n-\frac 1p (\widehat {D_0}f)_{n-1}\right)\\
&=p^{2n-2}[-p^2f_{n+1}+(1+p^2)f_n-f_{n-1}] \textrm{\, for any } n\geq 1.
\end{aligned}
\end{equation*}

We remark at this point that we could equivalently study the spectrum of $\widehat {D_0} \widehat {D_0^*}$; however the equations for the latter operator are not any simpler than the formulas for $\widehat {D_0^*}\widehat {D_0}$. Obviously, with the absence of kernels, the eigenvalue equations for both operators yield the same eigenvalues.  
\smallskip

The problem is now to solve the following eigenvalue equations  for $\widehat {D_0^*}\widehat {D_0}$: 

\begin{equation}\label{EVE}
\begin{aligned}
p^{2n-2}[-p^2f_{n+1}+(1+p^2)f_n-f_{n-1}]&=\lambda f_n;\textrm{ for } n\geq 1\\
f_0-f_1&=\lambda f_0,\\
\end{aligned}
\end{equation}
with $f_n \in \ell^2(\mathbb Z_{\geq 0})$. 

The key step in solving the system of equations \eqref{EVE} is the following result which asserts that all eigenvectors of $\widehat {D_0^*}\widehat {D_0}$ take the special exponential sum form $f_n=\sum_{k=0}^\infty c(k)p^{-nk}$ with rapidly decaying coeffients $c(k)$. This result is a form of elliptic regularity of the operator $\widehat {D_0^*}\widehat {D_0}$.

\begin{theo}\label{EVForm}
Let $\{f_{n}(\lambda)\}$ be an eigenvector of $\widehat {D_0^*}\widehat {D_0}$ with eigenvalue $\lambda$. Then the following statements are true.
\begin{enumerate}
\item The sequence $\{f_n(\lambda)\}_{n=0}^{\infty}$ belongs to $\ell^1(\mathbb Z_{\geq 0})$.
\item The eigenvector $f_n(\lambda)$ can be uniquely expressed in the form 
\begin{equation}\label{expform}
f_{n}(\lambda)=\sum_{k=1}^{\infty} c(2k)p^{-2nk},
\end{equation}
where the coefficients $c(2k)$ decay exponentially in $k$.
\item The coefficients $c(2k)$ satisfy the equations
\begin{equation*}c(2)=\left(\frac{\lambda}{1-p^{-2}}\right)\sum_{k=0}^{\infty} f_k,
\end{equation*}
and, for $ k\geq 2$,  
\begin{equation}\label{coeff}
c(2k)=\left(\frac{-\lambda}{1-p^{-2}}\right)^{k-1}\frac{c(2)p^{k(k-1)}(p^2-1)^{k-2}}{(p^4-1)^2(p^6-1)^2\ldots(p^{2k-2}-1)^2 (p^{2k}-1)}.
\end{equation}
\item If the remainder  $r_n(2N)$ is defined by the formula:
\begin{equation*}
f_n(\lambda)=c(2)p^{-2n}+c(4)p^{-4n}+c(6)p^{-6n}+ \ldots + c(2N-2)p^{-(2N-2)n}+ r_n(2N),
\end{equation*}
then $\{r_n(2N)\}_{n=0}^{\infty}\rightarrow 0$ as $N \rightarrow \infty$ in the $\ell^1$ norm.
\end{enumerate}
\end{theo}
\smallskip

\begin{proof}
The main idea of the proof is to rewrite the equations \eqref{EVE} in an integral equation form and then use it iteratively to produce the solution. To this end we regroup the terms in the first equation of system \eqref{EVE} above to obtain: 
\begin{equation*}
\left(f_n-f_{n-1}\right)-p^2 \left(f_{n+1}-f_n\right)=\lambda p^{2-2n} f_n.
\end{equation*}

Using the notation $\Delta f_n:=f_{n+1}-f_n$, we can then rewrite the system of equations \eqref{EVE} as follows. 
\begin{equation}\label{EVE1}
\begin{aligned}
\Delta f_n&=p^{-2}\left(\Delta f_{n-1}-\lambda p^{2-2n}f_n \right)\, \textrm{ for } n \geq 1\\
\Delta f_0&=f_1-f_0=-\lambda f_0.\\
\end{aligned}
\end{equation}

Iteratively, with the help of equations \eqref{EVE1}, we obtain the following formula for $\Delta f_n$:
\begin{equation}\label{Deltafn}
\Delta f_n=-p^{-2n}\lambda (f_0+f_1+\ldots +f_n), \, n \geq 0.
\end{equation}

Equation \eqref{Deltafn} is a one-step linear difference equation, so it has one-parameter family of solutions. However, since we are looking for the solution in the Hilbert space we need to choose one that vanishes at infinity. This leads to the following formula for $f_n$:
\begin{equation*}
f_n=\sum_{l=n}^{\infty} \lambda p^{-2l}\sum_{k=0}^l f_k.
\end{equation*} 

Interchanging the summation indices of the above formula we obtain:
\begin{equation}\label{E2fn}
\begin{aligned}
f_n&=\sum_{k=0}^{n} \sum_{l=n}^{\infty}\lambda p^{-2l} f_k+\sum_{k=n+1}^{\infty} \sum_{l=k}^{\infty}\lambda p^{-2l} f_k\\
&=\frac{\lambda}{(1-p^{-2})}\left[p^{-2n}\sum_{k=0}^n f_k+\sum_{k=n+1}^{\infty}p^{-2k} f_k\right].
\end{aligned}
\end{equation}

Thus we can estimate:
\begin{equation*}
\sum_{n=0}^{\infty}|f_n| \leq \frac{\lambda}{(1-p^{-2})}\left[\sum_{n=0}^{\infty}p^{-2n}\sum_{k=0}^n |f_k|+\sum_{n=0}^{\infty}\sum_{k=n+1}^{\infty}p^{-2k} |f_k|\right].
\end{equation*}

By interchanging the summation indices in the first sum above and evaluating the sum over $n$ we obtain:
\begin{equation*}
\sum_{n=0}^{\infty}p^{-2n}\sum_{k=0}^n |f_k| = \sum_{k=0}^{\infty} |f_k| \left(\frac{p^{-2k}}{1-p^{-2}}\right).
\end{equation*}

Using Cauchy-Schwartz inequality and the fact that $f \in \ell^2(\mathbb N)$ we conclude that this sum is bounded:
\begin{equation*}
\sum_{k=0}^{\infty}|f_k|\left(\frac{p^{-2k}}{1-p^{-2}}\right)\leq \left(\frac{1}{1-p^{-2}}\right)\sqrt{\frac{1}{1-p^{-4}}}\,\|f\|_2 < \infty.
\end{equation*}

Notice that for the second sum we have:
\begin{equation*}
\sum_{n=0}^{\infty}\sum_{k=n+1}^{\infty}p^{-2k} |f_k|=\sum_{n=0}^{\infty}p^{-2n}\sum_{l=1}^{\infty}p^{-2l}|f_{n+l}|.
\end{equation*}

Once again using Cauchy-Schwartz inequality we see that the second sum is finite:
\begin{equation*}
\sum_{n=0}^{\infty}p^{-2n}\sum_{l=1}^{\infty}p^{-2l}|f_{n+l}| \leq \sqrt{\frac{p^{-4}}{1-p^{-4}}}\left(\frac1{1-p^{-2}}\right)\|f\|_2< \infty .
\end{equation*}
This verifies that $\{f_n\}\in \ell^1(\mathbb N)$.

\bigskip

To prove the second part of Theorem \ref{EVForm},  we observe that equation \eqref{E2fn} gives; 
\begin{equation*}
f_n=\frac{\lambda p^{-2n}}{(1-p^{-2})}\left[\sum_{k=0}^{\infty} f_k-\sum_{k=n+1}^{\infty} f_k+\sum_{l=1}^{\infty}p^{-2l} f_{n+l}\right].
\end{equation*}

Rearranging the terms on the right hand side of the equation to isolate the coefficient of $p^{-2n}$ we get:
\begin{equation}\label{E3fn}
f_n=\left(\frac{\lambda}{1-p^{-2}}\sum_{k=0}^{\infty} f_k\right)p^{-2n}-\lambda p^{-2n} \sum_{l=1}^{\infty}\left(\frac{1-p^{-2l}}{1-p^{-2}}\right)f_{n+l},
\end{equation}
from which we extract the coefficient  
\begin{equation*}
c(2):=\left(\frac{\lambda}{1-p^{-2}}\right)\sum_{k=0}^{\infty} f_k.
\end{equation*}
 Notice that $c(2)$ is well defined due to part (1).

Recursively applying this formula for $f_n$ on the right hand side of equation \eqref{E3fn} we obtain:
\begin{equation*}
f_n=c(2)p^{-2n}-\lambda p^{-2n}\sum_{l=1}^{\infty}\left(\frac{1-p^{-2l}}{1-p^{-2}}\right)\left(c(2)p^{-2n-2l}-\lambda p^{-2n-2l} \sum_{k=1}^{\infty}\left(\frac{1-p^{-2k}}{1-p^{-2}}\right) f_{n+l+k}\right).
\end{equation*}

Once again we rearrange the terms to extract the coefficient $c(4)$ of $p^{-4n}$.

\begin{equation*}
\begin{aligned}
f_n=c(2)p^{-2n}&+\left(-\lambda c(2)\sum_{l=1}^{\infty}\left(\frac{p^{-2l}-p^{-4l}}{1-p^{-2}}\right)\right)p^{-4n} \\
&+  \frac{\lambda^2}{p^{4n}}\sum_{l=1}^{\infty}\left(\frac{p^{-2l}-p^{-4l}}{1-p^{-2}}\right)\sum_{k=1}^{\infty}\left(\frac{1-p^{-2k}}{1-p^{-2}}\right) f_{n+l+k}.
\end{aligned}
\end{equation*}

This gives:
\begin{equation*}
c(4)=\frac{-\lambda c(2)}{(1-p^{-2})}\cdot\frac{p^2}{(p^4-1)}\ .
\end{equation*}

By repeatedly applying this process we can obtain an expansion of $f_n$ in powers of $p^{-2n}$, provided the remainder $r_n(2N)$ goes to zero as $N\to\infty$. We prove a stronger $\ell^1$ estimate on $r_n(2N)$ below, implying the pointwise convergence needed for the existence of the expansion of $f_n$.

\smallskip

Using induction we readily establish that the coefficients $c(2k)$ of this expansion are in general given by the formula:

\begin{equation}\label{Fcoeff}
c(2k)=\left(\frac{-\lambda}{1-p^{-2}}\right)^{k-1}\frac{c(2)p^{k(k-1)}(p^2-1)^{k-2}}{(p^4-1)^2(p^6-1)^2\ldots(p^{2k-2}-1)^2 (p^{2k}-1)}\ .
\end{equation}
\smallskip

Next we estimate the growth of the coefficients $c(2k)$. Simplifying the formula for $c(2k)$ we obtain:

\begin{equation*}
\begin{aligned}
c(2k)&=\frac{c(2)\lambda^{k-1}}{p^{(k-2)(k-1)}(1-\frac1{p^{2}})(1-\frac 1{p^{4}})^2(1-\frac1{p^{6}})^2\ldots (1-\frac1{p^{2k-2}})^2(1-\frac1{p^{2k}})}\\
&= \frac{\lambda^k (1-\frac1{p^{2k}}) \sum_{m=0}^{\infty}f_m }{p^{(k-2)(k-1)}(1-\frac1{p^{2}})^2(1-\frac 1{p^{4}})^2(1-\frac1{p^{6}})^2\ldots (1-\frac1{p^{2k-2}})^2(1-\frac1{p^{2k}})^2}\ .
\end{aligned}
\end{equation*}

Since $\prod_{i=1}^k (1-p^{-2i})^2 \geq \prod_{i=1}^{\infty} (1-p^{-2i})^2$ and the infinite product is a finite nonzero number,  we obtain the following estimate for the coefficients $c(2k)$:

\begin{equation*}
|c(2k)|\leq \frac{|\lambda|^{k}\|f\|_1}{p^{(k-2)(k-1)}\prod_{i=1}^{\infty} (1-p^{-2i})^2}\ , 
\end{equation*}
which shows that they decay exponentially. This establishes both the second and the third part of the theorem.

\smallskip

Finally we estimate the remainder term $r_n(2N)$. Using induction it is easily established that the remainder $r_n(2N)$ satisfies the following summation formula:

\begin{equation*}
\begin{aligned}
r_n(2N)=&\frac{(-\lambda)^N}{p^{2nN}}\sum_{l_1=1}^{\infty}\sum_{l_2=1}^{\infty} \ldots \sum_{l_N=1}^{\infty}\frac{({p^{-(2N-2)l_1}}-p^{-2Nl_1})}{1-p^{-2}}\frac{(p^{-(2N-4)l_2}-p^{-(2N-2)l_2})}{1-p^{-2}}\\
&\ldots\frac{(1-p^{-2l_N})}{1-p^{-2}}
f_{n+l_1+l_2+\ldots+l_N}.
\end{aligned}
\end{equation*}

 Estimating the $\ell^1$ norm we see that:

\begin{equation}\label{error}
\begin{aligned}
\sum_{n=1}^{\infty}|r_n(2N)|\leq &\sum_{l_1=1}^{\infty}\sum_{l_2=1}^{\infty} \ldots \sum_{l_N=1}^{\infty}\frac{|\lambda|^N}{p^{2nN}(1-p^{-2})^N}({p^{-(2N-2)l_1}}-p^{-2Nl_1})(p^{-(2N-4)l_2}-p^{-(2N-2)l_2})\\
&\ldots (1-p^{-2l_N})|f_{n+l_1+l_2+\ldots+l_N}|.\\
\end{aligned}
\end{equation}

Notice that the term $\sum_{l_N=1}^{\infty}(1-p^{-2l_N})|f_{n+l_1+l_2+\ldots+l_N}|$ can be estimated as follows:

\begin{equation*}
\begin{aligned}
\sum_{l_N=1}^{\infty}(1-p^{-2l_N})|f_{n+l_1+l_2+\ldots+l_N}|&\leq \sup_{l_N \geq 1}(1-p^{-2l_N})\sum_{l_N=1}^{\infty}|f_{n+l_1+l_2+\ldots+l_N}|\\
&\leq \|f\|_1
\end{aligned}
\end{equation*}
where the last line can be justified by changing the summation index in the previous line appropriately. Moreover, we can explicitly calculate each sum that appears in formula  
\eqref{error}. For example:

\begin{equation*}
\begin{aligned}
\sum_{n=1}^{\infty}p^{-2Nn}&=\frac{p^{-2N}}{1-p^{-2N}},\\
\sum_{l_1=1}^{\infty}({p^{-(2N-2)l_1}}-p^{-2Nl_1})&=\frac{(1-p^{-2})p^{-(2N-2)}}{(1-p^{-(2N-2)})(1-p^{-2N})},\\
\sum_{l_2=1}^{\infty}(p^{-(2N-4)l_2}-p^{-(2N-2)l_2})&=\frac{(1-p^{-2})p^{-(2N-4)}}{(1-p^{-(2N-4)})(1-p^{-(2N-2)})},
\end{aligned}
\end{equation*}
and so on. Substituting all these values into the formula \eqref{error} we get the following estimate:

\begin{equation*}
\sum_{n=1}^{\infty}|r_n(2N)|\leq \frac{|\lambda|^N}{(1-p^{-2})^N}\cdot \frac{p^{-2N}}{1-p^{-2N}} \cdot \frac{(1-p^{-2})p^{-(2N-2)}}{(1-p^{-(2N-2)})(1-p^{-2N})}\ldots \frac{(1-p^{-2})p^{-2}}{(1-p^{-4})(1-p^{-2})}\ .
\end{equation*}

Simplifying this expression we obtain:

\begin{equation*}
\sum_{n=1}^{\infty}|r_n(2N)|\leq \frac{|\lambda|^N}{p^{N(N+1)}\prod_{k=1}^N (1-p^{-2k})^2} \leq \frac{|\lambda|^N}{p^{N(N+1)}\prod_{k=1}^{\infty} (1-p^{-2k})^2}. 
\end{equation*}
Since $\prod_{k=1}^N (1-p^{-2k})^2 < \infty$ we see that $\sum_{n=1}^{\infty}|r_n(2N)| \rightarrow 0$ as $N \rightarrow \infty$. 

\smallskip

Finally we will prove uniqueness of the expansion of $f_n(\lambda)$. Consider the analytic function $f(z)= \sum_{k=1}^{\infty}c(2k)z^k$. From the above estimate of the coefficients $c(2k)$ we see that the radius of convergence $R$ of the power series for $f$ is given by:

\begin{equation*}
\frac 1R= \limsup_{k \rightarrow \infty} \sqrt[k]{|c(2k)|} \leq |\lambda| \;\limsup_{k \rightarrow \infty}\sqrt[k]{\frac{\|f\|_1}{\prod_{i=1}^{\infty} (1-p^{-2i})^2}} \cdot \frac1{p^{k-3+2/k}}.
\end{equation*}

Therefore, $R=\infty$ and the function $f(z)$ is entire. Therefore, in particular:
\begin{equation*}
f(p^{-2n})=\sum_{k=1}^{\infty}c(2k)p^{-2nk}=f_n(\lambda),
\end{equation*}
and so the coefficients $c(2k)$ are uniquely determined by $f_n(\lambda)$, because an analytic function is completely determined by its values on a convergent sequence of points, \cite{AL}.

\end{proof}
\bigskip

\begin{rem} The collection of $\ell^2$ functions with a power series representation of the form \eqref{expform} is fairly restrictive which is clear from the fact that $\lim_{n \rightarrow \infty} p^{2n}f(n)= c(2)$.
 It can be easily shown that the set of $\ell^2$ functions with this power series representation is dense in the space of all $\ell^2$ functions.
\end{rem}

The difficult part already completed, we can now state our main theorem. 
\begin{theo}
The spectrum of the operator $\widehat {D_0^*} \widehat {D_0}$ consists of simple eigenvalues $\{\lambda_n\}$ which are the roots of the $q$-hypergeometric function  $\lambda\mapsto\ _1\phi_1\left(\substack {0\\q};q,\lambda \right)$, with $q=\frac1{p^2}$.
\end{theo}

\begin{proof} 
Substituting $f_n=\sum_{k=1}^\infty c(2k)p^{-2nk}$ and formula  \eqref{Fcoeff} into the initial condition of system \eqref{EVE} and dividing throughout by $c(2)$ we obtain the following:
 
\begin{equation}\label{NEVE}
\frac{1}{p^2}+\lambda -1+ \sum_{k=2}^\infty \frac{\left(p^{-2k}+\lambda -1\right) \lambda^{k-1}p^{2k-2}} {\prod_{j=2}^k(1-p^{2j})\left(1-\cfrac1{p^{2j-2}}\right)} =0.
\end{equation}

The infinite sum on the left hand side of the above equation, call it $S_1$, can be simplified by first breaking it up into two terms, extracting some terms and then recombining as follows: 

\begin{equation*}
\begin{aligned}
S_1&=\sum_{k=2}^\infty \frac{\left(p^{-2k} -1\right) \lambda^{k-1}p^{2k-2}} {\prod_{j=2}^k(1-p^{2j})\left(1-\cfrac1{p^{2j-2}}\right)} + \sum_{k=3}^\infty \frac{\lambda^{k-1}p^{2k-4}}{\prod_{j=2}^{k-1}(1-p^{2j})\left(1-\cfrac1{p^{2j-2}}\right)}\\
&=\frac{\left(\frac1{p^4}-1\right)\lambda p^2}{\left(1-p^4\right)\left(1-\frac1{p^2}\right)}+ \sum_{k=3}^\infty \frac{p^{2k-4}\lambda^{k-1}\left[p^2(p^{-2k}-1)+(1-p^{2k})\left(1-\frac1{p^{2k-2}}\right)\right]}{\prod_{j=2}^{k}(1-p^{2j})\left(1-\cfrac1{p^{2j-2}}\right)}\ .\\
\end{aligned}
\end{equation*}

Using the substitution $q=\frac1{p^2}$ equation \eqref{NEVE} can be written as
\begin{equation*}
(q-1)+ \frac{\lambda}{(1-q)}+ \sum_{k=3}^\infty \frac{q^{2-k}\lambda^{k-1}}{\prod_{j=2}^{k-1}(1-q^{-j})\prod_{j=2}^{k}\left(1-q^{j-1}\right)}=0.
\end{equation*}
Notice that at $k=2$ the expression
\begin{equation*} 
\frac{q^{2-k}\lambda^{k-1}}{\prod_{j=2}^{k-1}(1-q^{-j})\prod_{j=2}^{k}\left(1-q^{j-1}\right)}
\end{equation*} yields the value $\frac{\lambda}{(1-q)}$. Thus the above equation is in fact equal to:

\begin{equation*}
(q-1)+ \sum_{k=2}^\infty \frac{q^{2-k}\lambda^{k-1}}{\prod_{j=2}^{k-1}(1-q^{-j})\prod_{j=2}^{k}\left(1-q^{j-1}\right)}=0.
\end{equation*}

Now we rearrange the terms in the infinite sum in order to compare it with the hypergeometric function $_1\phi_1\left(\substack {0\\q};q,\lambda \right)$.

\begin{equation*}
\begin{aligned}
(q-1)+ \sum_{k=2}^\infty \frac{q^{2-k}\lambda^{k-1}}{\prod_{j=2}^{k-1}(1-q^{-j})\prod_{j=2}^{k}\left(1-q^{j-1}\right)}&=(q-1)+ \sum_{k=2}^\infty \frac{(-1)^{k}\lambda^{k-1}(1-q)q^{\frac{(k-2)(k-1)}{2}}}{\prod_{j=1}^{k-1}(1-q^{j})^2}\\
&=(q-1)- \sum_{k=1}^\infty \frac{(-1)^{k}\lambda^{k}(1-q)q^{\frac{k(k-1)}{2}}}{\prod_{j=1}^{k}(1-q^{j})^2}\\
&=1+ \sum_{k=1}^\infty \frac{(-1)^{k}\lambda^{k}q^{\frac{k(k-1)}{2}}}{\prod_{j=1}^{k}(1-q^{j})^2}\ .
\end{aligned}
\end{equation*}

By using the notation 
\begin{equation*}
(a;q)_n=(1-a)(1-aq)\ldots(1-aq^{n-1})
\end{equation*}
 and the above computation, we can rewrite the eigenvalue equation as
\begin{equation*}
1+ \sum_{k=1}^\infty \frac{(-1)^{k}\lambda^{k}q^{\frac{k(k-1)}{2}}}{(q;q)_k^2}=0.
\end{equation*}

The function $_1\phi_1$ of four variables $a_0, b_1,q,z$ is defined as
\begin{equation*}
_1\phi_1\left(
\begin{array}{c@{}c@{}c}
\begin{array}{c}
 a_0\\
 b_1\\
\end{array} ;&\ q^2 ,&\ z\\
\end{array}\right)=\sum_{n=0}^\infty \frac{(a_0;q^2)_n}{(q^2;q^2)_n(b_1;q^2)_n}(-1)^nq^{2\binom{n}{2}}z^n.
\end{equation*}

Thus, if $\lambda$ is an eigenvalue, we get:
\begin{equation*}
_1\phi_1\left(\substack {0\\q};q,\lambda \right)=1+ \sum_{k=1}^\infty \frac{(-1)^{k}\lambda^{k}q^{\frac{k(k-1)}{2}}}{(q;q)_k^2}=0,
\end{equation*}
showing that the eigenvalues of the operator $\widehat {D_0^*}\widehat {D_0}$ are the roots of the above $q$ - hypergeometric function.
Conversely, the above calculation shows that given a root $\lambda$ of $\lambda\mapsto\ _1\phi_1\left(\substack {0\\q};q,\lambda \right)$ the formula \eqref{expform} with arbitrary $c(2)$ and
other coefficients $c(2k)$ given by \eqref{coeff} gives, up to a constant, the unique eigenvector of $\widehat {D_0^*}\widehat {D_0}$ corresponding to eigenvalue $\lambda$. By the analysis in \cite{KMR} the whole spectrum of $\widehat {D_0^*}\widehat {D_0}$ consists of such eigenvalues.
\end{proof}
\smallskip


\maketitle
\section{Spectral properties}

\subsection{Spectrum of $D^*D$}
Computation of the spectrum of $D^*D$ is based on decomposition \eqref{decomD*D} and the analysis of the spectrum of $\widehat {D_0^*}\widehat {D_0}$ in the previous section.
\smallskip

\begin{theo}
Let $\{\lambda_n\}$ be the eigenvalues of the operator $\widehat {D_0^*}\widehat {D_0}$ and let $\widehat D_g^*\widehat D_g$ be as in formula \eqref{Dg}. Then,
\begin{enumerate}
\item The spectrum of $\wh {D_g^* }\wh {D_g}$ consists of simple eigenvalues $\{p^{2m}\lambda_n\}$  i.e., $\sigma (\wh {D_g^* }\wh {D_g})= \bigcup_n \{p^{2m} \lambda_n\}$. 
\item $\sigma ( D^*  D)= \sigma (\wh {D^*} \wh D)= \bigcup_{m,n} \{p^{2m} \lambda_n\}$. Moreover, each eigenvalue of $\wh {D^*} \wh D$ occurs with multiplicity $p^m(1-\frac 1p)$.

\end{enumerate}
\end{theo}

\begin{proof}

The above results follow directly from the decomposition \eqref{decomD*D}.  Since the number of different values of $r$ less than $p^m$ that are relatively prime to $p$ is equal to $p^m-p^{m-1}$, each eigenvalue of $\widehat D^*\widehat D$ in $H$ has multiplicity $p^m(1-\frac 1p)$.  
\end{proof}

\begin{cor}
The operator $(D^*D)^{-1}$ is a $s$-th Schatten class opeartor for all $s \geq 1$.
\end{cor}

\begin{proof}
From the decomposition \eqref{decomD*D} we see that:

\begin{equation*}
(D^*D)^{-s}= \bigoplus_{\frac r{p^m}\in \mathcal G_p}p^{-2ms}\,(D_0^*D_0)^{-s},
\end{equation*}
from which we compute the following trace:
\begin{equation*}
\begin{aligned}
\Tr (D^*D)^{-s}&=\sum_{\frac r{p^m}\in \mathcal G_p}p^{-2ms}\, \Tr (D_0^*D_0)^{-s}\\
&=\sum_{m=0}^{\infty}\sum_{\substack {0 \leq r < p^m \\ p \nmid r}}p^{-2ms}\, \Tr (D_0^*D_0)^{-s}.
\end{aligned}
\end{equation*}

Since the number of nonnegative $r$'s less than $p^m$ and relatively prime to $p$ is equal to the Euler number of $p^m$, we can compute the sum over $m$ provided that $s > \frac 12$:
\begin{equation}
\begin{aligned}\label{traceD*D}
\Tr (D^*D)^{-s}&=\sum_{m=0}^{\infty} (p^m-p^{m-1})p^{-2ms}\, \Tr (D_0^*D_0)^{-s}\\
&=(1-\frac 1p)\left(\frac1{1-p^{1-2s}}\right) \,\Tr (D_0^*D_0)^{-s}.
\end{aligned}
\end{equation}

From \cite{ABC} we have that $\lambda_n \leq p^n$, so we can estimate the trace $\Tr (D_0^*D_0)^{-s}=\sum_{n=0}^{\infty}(\lambda_n)^{-s}$ as follows, provided $s>0$:

\begin{equation*}
\Tr (D_0^*D_0)^{-s}=\sum_{n=0}^{\infty}(\lambda_n)^{-s} \leq \sum_{n=0}^{\infty} p^{-ns}=\frac 1{1-p^{-s}}.
\end{equation*}

Summing up this information we see that,

\begin{equation*}
\Tr (D^*D)^{-s}\leq \left(1-\frac 1p \right)\left(\frac1{1-p^{1-2s}}\right)\left( \frac 1{1-p^{-s}}\right)
\end{equation*}
whenever $s> \frac 12$. Thus for any $s \geq 1$ the $s$-th Schatten norm of $(D^*D)^{-1}$ is finite.
\end{proof}
\bigskip

\bigskip

\subsection{Analytic continuation of the zeta functions}

Using formula \eqref{traceD*D} we can express the zeta function associated with the operator $D^*D$, denoted $\zeta_D(s)$, in terms of $\zeta_{D_0}(s)$, the zeta function associated with the operator $D_0^*D_0$:

\begin{equation}\label{zeta}
\zeta_D(s)=(1-\frac 1p)\left(\frac1{1-p^{1-2s}}\right) \zeta_{D_0}(s).
\end{equation}

We now consider the analytic continuation of $\zeta_{D_0}(s)$.

\begin{theo}
$\zeta_{D_0}(s)$ is  holomorphic for $\Re s >0$ and can be analytically continued to a meromorphic function for $\Re s > -2$.
\end{theo}

\begin{proof}
To show that $\zeta_{D_0}(s)$ is holomorphic in the region $\Re s >0$ we estimate:

\begin{equation*}
|\zeta_{D_0}(s)|\leq \sum_{n=1}^{\infty}\left|\frac 1{\lambda_n^{\Re s + i \, \Im s}}\right|= \sum_{n=1}^{\infty} \frac 1{\lambda_n^{\Re s}}\ ,
\end{equation*}
since $\lambda_n^{-i\, \Im s}$ is unimodular. From \cite{ABC} we know that the eigenvalues $\lambda_n$ of $D_0^*D_0$ satisfy the following upper and lower bounds:

\begin{equation}\label{ulbounds}
p^n \left(1-\frac{p^{-2n}}{1-p^{-2n}}\right) < \lambda_n < p^n.
\end{equation}

Thus we get:
\begin{equation*}
|\zeta_{D_0}(s)| \leq \sum_{n=1}^{\infty} \frac 1{\left(p^n \left(1-\frac{p^{-2n}}{1-p^{-2n}}\right)\right)^{\Re s}}.
\end{equation*}

We have an elementary inequality:
\begin{equation*}
\frac{p^{-2n}}{1-p^{-2n}}= 1- \frac 1{p^{2n}-1} \geq \frac{p^2-2}{p^2-1},
\end{equation*}
which holds since the left-hand side is an increasing function of $n$,  while the right-hand side is its value at $n=1$. Therefore, we get:

\begin{equation*}
|\zeta_{D_0}(s)| \leq \left(\frac{p^2-1}{p^2-2}\right)^{\Re s} \,\sum_{n=1}^{\infty} \frac 1{p^{n\Re s}},
\end{equation*}
which is convergent for $\Re s >0$. Consequently, $\zeta_{D_0}(s)$ is holomorphic in $\Re s >0$. 

\smallskip

We now show that $\zeta_{D_0}(s)$ can be analytically continued to $\Re s > -2$. Since $\lambda_n$ behaves like $p^n$, the analytic continuation of $\zeta_{D_0}(s)$ will be achieved by a perturbative argument from the meromorphic function obtained from the zeta function by replacing  $\lambda_n$ with $p^n$. First we write:

\begin{equation*}
p^{-ns}- \lambda_n^{-s}=e^{-sn\ln p}-e^{-s \ln(\lambda_n)}=\int_{-\ln(\lambda_n)}^{-n\ln p}\frac d{dt}e^{ts} dt=s\int_{-\ln(\lambda_n)}^{-n\ln p}e^{ts} dt.
\end{equation*}

Thus, we obtain:

\begin{equation*}
|p^{-ns}- \lambda_n^{-s}|\leq |s|\int_{-\ln(\lambda_n)}^{-n\ln p}|e^{ts}| dt=|s|\int_{-\ln(\lambda_n)}^{-n\ln p}e^{t \Re s} dt.
\end{equation*}
In this integral we can estimate the integrand  by its maximum on the interval of integration $[-\ln(\lambda_n),-n\ln p]$ to arrive at the following estimate:

\begin{equation*}
|p^{-ns}- \lambda_n^{-s}|\leq \begin{cases}
															|s|(n \ln p - \ln(\lambda_n))e^{-n \ln p \Re s} & \textrm{ if } \Re s \leq 0 \\
															|s|(n \ln p - \ln(\lambda_n))e^{- \ln(\lambda_n) \Re s} & \textrm{ if } \Re s> 0. \\
															\end{cases}
\end{equation*}

Inequality \eqref{ulbounds} implies that:
\begin{equation*}
\ln\left(\frac {p^n}{\lambda_n}\right)< -\ln \left(1-\frac{p^{-2n}}{1-p^{-2n}}\right)= \frac{p^{-2n}}{1-p^{-2n}} \sum_{k=0}^{\infty} \frac 1{k+1} \left(\frac{1}{p^{2n}-1}\right)^k. 
\end{equation*}
Since $1-p^{-2n}> \frac 12$ for $n\geq 1$, we can estimate the above as:

\begin{equation*}
\ln\left(\frac {p^n}{\lambda_n}\right)< 2p^{-2n}\sum_{k=0}^{\infty} \frac 1{k+1} \left(\frac{1}{p^2-1}\right)^k= -2p^{-2n}\ln \left(1- \frac 1{p^2-1}\right).
\end{equation*}

Consequently, if $\Re s \leq 0$, we have: 
\begin{equation*}
|p^{-ns}- \lambda_n^{-s}|\leq -2 |s|\ln \left(1- \frac 1{p^2-1}\right)p^{-2n}p^{-n \Re s}.
\end{equation*}

This lets us estimate the difference of the series as follows:

\begin{equation*}
\left|\sum_{n=1}^{\infty}  (p^{-ns}- \lambda_n^{-s})\right| \leq -2 \ln \left(1- \frac 1{p^2-1}\right)|s| \sum_{n=1}^{\infty}p^{-n(2+ \Re s)}.
\end{equation*}

The series $\sum_{n=1}^{\infty}p^{-n(2+ \Re s)}$ is convergent for $\Re s > -2$ hence, by the Weierstrass $M$ test, the series $\sum_{n=1}^{\infty}  (p^{-ns}- \lambda_n^{-s})$ converges uniformly for $\Re s > -2$ and hence it is analytic for $\Re s > -2$.

Moreover, since
\begin{equation*}
\sum_{n=1}^{\infty}p^{-ns}= \frac {p^{-s}}{1-p^{-s}}
\end{equation*}
is meromorphic in the complex plane with poles at $s= \frac {2 \pi i k}{\ln p}$,  $k \in \mathbb Z$, we obtain that the zeta function
$\sum_{n=1}^{\infty}\lambda_n^{-s}$ for meromorphic for $\Re s > -2$ with the above mentioned poles. 

\end{proof}

\begin{cor}
$\zeta_D(s)$ is meromorphic for $\Re s > -2$ with poles at $s=\frac{2 \pi ik}{\ln p}$,  and $s=\frac 12\left(1-\frac{2 \pi ik}{\ln p}\right)$, where $k\in \mathbb Z$.
\end{cor}

\begin{proof}
The proof of this corollary follows from the theorem above and equation \eqref{zeta}.
\end{proof}


\section{Appendix}

In this section we record some basic properties and identities satisfied by the $q$ - hypergeometric function $_1 \phi_1$ we encountered in section 4. More on $q$ - hypergeometric functions can be found in \cite{GR}. We start with the general definition of these type of functions:
\begin{equation*}
_{r+1}\phi_s\left(
\begin{array}{c@{}c@{}c}
\begin{array}{c}
a_0,a_1,\ldots,a_r\\
b_1,b_2,\ldots,b_s\\
\end{array} ;&\ q ,&\ z\\
\end{array}\right)%
=\sum_{n=0}^\infty \frac{(a_0;q)_n(a_1;q)_n\ldots (a_r;q)_n}{(q;q)_n(b_1;q)_n\ldots (b_s;q)_n}\left((-1)^nq^{\binom{n}{2}}\right)^{s-r}z^n
\end{equation*}
where $b_j \neq q^{-n}$ for any $j,n$. 

 Here we  used the notation $(a;q)_n=(1-a)(1-aq)\ldots(1-aq^{n-1})$. We remark that $(a;q)_{\infty}=\prod_{j=0}^{\infty}(1-aq^j)$. When $s>r$ the above series converges for all $z$ while it converges for $|z|<1$ when $s=r$.

We are interested in the special case where $r=0, s=1$ and $a=0, b=q$, which leads to the formula:
\begin{equation*}
_1\phi_1\left(
\begin{array}{c@{}c@{}c}
\begin{array}{c}
 0\\
 q\\
\end{array} ;&\ q ,&\ z\\
\end{array}\right)=\sum_{n=0}^\infty \frac{1}{(q;q)_n(q;q)_n}(-1)^nq^{\binom{n}{2}}z^n.
\end{equation*}

The function $_1 \phi_1$ satisfies the Cauchy's sum:
\begin{equation*}
_1\phi_1\left(
\begin{array}{c@{}c@{}c}
\begin{array}{c}
 a\\
 b\\
\end{array} ;&\ q ,&\ b/a\\
\end{array}\right)=\frac{(b/a;q)_{\infty}}{(b;q)_{\infty}}.
\end{equation*}

 In \cite{ABC} the authors investigated the roots of the third Jackson $q$-Bessel function:
 \begin{equation*}
 J_{\nu}^{(3)}(z;q):= z^{\nu} \frac{(q^{\nu +1};q)_{\infty}}{(q;q)_{\infty}}\; {_1\phi_1}\left(
\begin{array}{c@{}c@{}c}
\begin{array}{c}
0\\
 q^{\nu+1}\\
\end{array} ;&\ q ,&\ qz^2\\
\end{array}\right),
 \end{equation*} 
where $0 < q <1$ and $z$ is a complex parameter. It is known that this function has infinitely many zeros, each of multiplicity one, all of them real. When $\nu=0$ we that the third Jackson $q$-Bessel function equals the function $_1 \phi_1$ which we used in this paper. 

We record the following transformation property of $_1 \phi_1$:
\begin{equation*}
_1\phi_1\left(
\begin{array}{c@{}c@{}c}
\begin{array}{c}
 0\\
 b\\
\end{array} ;&\ q ,&\ z\\
\end{array}\right)=\frac{(z;q)_{\infty}}{(b;q)_{\infty}}\ {_1\phi_1}\left(
\begin{array}{c@{}c@{}c}
\begin{array}{c}
 0\\
 z\\
\end{array} ;&\ q ,&\ b\\
\end{array}\right).
\end{equation*}
 
Starting with this transformation the authors in \cite{ABC} deduce that if $q< (1-q)^2$ then the positive roots $\omega_k(q)$, $k=1,2,3, \ldots$,  of $J_0^{(3)}(z;q)$, arranged in the increasing order satisfy the following:

\begin{equation*}
q^{-k/2+ \alpha_k(q)}< \omega_k(q)< q^{-k/2},
\end{equation*}
where 
\begin{equation*}
\alpha_k(q)=\frac{\log \left(1- \frac{q^k}{1-q^k}\right)}{\log q}.
\end{equation*}
In particular, this gives the asymptotic behavior $\omega_k \sim q^{-k/2}$ as $k \rightarrow \infty$. 
Additionally, those results give an upper and lower bound \eqref{ulbounds} for the roots of the specific $_1 \phi_1$ function needed in this paper.

\end{document}